\documentclass[11pt,twoside]{amsart}
\usepackage{latexsym}
\usepackage{amssymb,amsbsy,amsmath,amsfonts,amssymb,amscd,amsthm}
\usepackage{graphicx}
\usepackage{hyperref}
\usepackage{pdfsync}
\usepackage{cancel}
\usepackage{color}
\newcommand{\commentout}[1]{}
\newcommand{\ind}[1]{\mbox{ \bf 1}_{\{#1\} }}

\newcommand {\wh}[1] {{\widehat #1}}

\newcommand{\E}{\mathbb{E}}
\newcommand{\R}{\mathbb{R}}

\newcommand {\e}  {\varepsilon}

\newcommand {\lb} {\lambda}

\newcommand {\sgn} { {\rm sgn} }
\newcommand {\dv}  { {\rm div} }
\newcommand {\f}   {\frac}
\newcommand {\p}   {\partial}
\newcommand{\dis}{\displaystyle}
\newcommand{\beq}{\begin{equation}}
\newcommand{\beqa}{\begin{eqnarray}}
\newcommand{\bea} {\begin{array}{rl}}
\newcommand{\beqan}{\begin{eqnarray*}}
\newcommand{\eeq}{\end{equation}}
\newcommand{\eeqa}{\end{eqnarray}}
\newcommand{\eeqan}{\end{eqnarray*}}
\newcommand{\eea} {\end{array}}
\newtheorem{theorem}{Theorem}[section]
\newtheorem{lemma}[theorem]{Lemma}
\newtheorem{definition}[theorem]{Definition}

\newcommand{\cqfd}{{ \hfill
                       {\unskip\kern 6pt\penalty 500
                       \raise -2pt\hbox{\vrule\vbox to 6pt{\hrule width 6pt
                       \vfill\hrule}\vrule} \par}   }}

\title{ {\bf{\small   Scalar conservation laws with rough (stochastic) fluxes; }}
\\[2pt]
{\small the spatially dependent case}}

\author{ Pierre-Louis Lions$^{1}$, Beno\^ \i t Perthame$^{2}$ and Panagiotis E. Souganidis$^{3,4}$}
\date{\today}

\begin{document}
\maketitle
\vspace*{1.0cm}
\pagenumbering{arabic}

\begin{abstract}

We continue the development of the theory of pathwise stochastic entropy solutions  for scalar conservation laws in $\R^N$ with quasilinear multiplicative ``rough path'' dependence by considering inhomogeneous  fluxes and a single rough path like, for example, a Brownian motion.
 Following our previous note where we considered spatially independent fluxes,
we introduce the notion of pathwise stochastic entropy solutions and prove that it is well posed, that is  we establish existence, uniqueness and continuous dependence in the form of a
(pathwise) $L^1$-contraction.
Our approach is motivated by  the theory of stochastic viscosity solutions, which was introduced and developed
by two of the authors, to study fully nonlinear first- and second-order stochastic pde with multiplicative noise. This  theory relies on special test functions constructed by inverting locally the flow of the stochastic characteristics. For conservation laws this is best implemented at the level of the kinetic formulation which we follow here. 

\end{abstract}

\pagestyle{plain} \vspace*{1.0cm}

\\[2mm]

\section{Introduction}
\label{sec:intro}
\noindent  
%
We continue the development of the theory of pathwise stochastic entropy solutions for scalar conservation laws in $\R^N$ with quasilinear multiplicative ``rough path'' dependence by considering inhomogeneous, that is spatially dependent,  fluxes and a single rough path. In our previous note \cite{LPSscl} we studied spatially independent fluxes and multiple paths.

\smallskip

\noindent Our approach is based  on the concepts and methods introduced in \cite{LPSscl} as well as by Lions and Souganidis  in \cite{LShjscras98, LShjscras98b} and extended by the same authors in \cite{LShjs, LSsemilinear, LSuniqueness, LShjscrasEDP} for the theory of pathwise stochastic viscosity solution of fully nonlinear first- and second-order stochastic pde including stochastic Hamilton-Jacobi equations. One of the fundamental tools of this theory is the class of test functions constructed by inverting locally, and at the level of test functions, the flow of the characteristics corresponding to the stochastic first-order part of the equation and smooth initial data. Such approach is best implemented for conservation laws using the kinetic formulation which we follow here.

\smallskip

\noindent  It is important to remark that throughout the paper we use the term ``rough path'' for a non differentiable time dependent function and we do not make any connection with the Lyons \cite{lyons} theory of rough paths since we are only considering a single path. 
We also note that to keep statements shorter,  we often write sscl for scalar stochastic conservation laws any time we refer to equations of the type of study here.

\smallskip

\noindent  Let
\beq\label{flux}
{\bf{A}} =(A_1,...,A_N) \in C^2(\R^N\times \R;\R^N)
\eeq
 be the flux and consider a single continuous ``rough path'', which may be, for example, a Brownian motion,
 \begin{equation}\label{path}
 W \in C([0,\infty);\R)   \ \text{ with } \ W(0)=0.
 \end{equation}
 
\smallskip

\noindent  We are interested in the sscl
\begin{equation}\label{xscl}
\begin{cases}
du + \dv {\bf{A}}(x, u) \circ dW = 0 \ \text{ in } \ \R^N\times(0,\infty), 
\\[2mm]
u=u^0 \quad \text{ on } \ \R^N \times \{ 0 \}.
\end{cases}
\end{equation}
Note that $u(x,t)= v(x, W(t))$, where $v$ solves a time-homogeneous equation $v_t + \dv {\bf{A}}(x, v)=0,$  is formally a solution but it can not be an entropy solution because the change of time is incompatible with the formation of shocks.

\smallskip

\noindent  Throughout the paper we adopt the notation and terminology of stochastic calculus. In general $du$ denotes some kind of time differential, while, in the case that $W$ is Brownian, it is the usual stochastic differential. Similarly in the general setting $\circ$ does not have any particular meaning and can be ignored, while in the stochastic setting 
it denotes the Stratonovich differential. The need to use the latter stems from the fact that we are developing a theory which is closed (stable) on paths in the local uniform topology and, in this context, Stratanovich is relevant.
That a pathwise theory is more appropriate to study \eqref{xscl} is also justified from the fact that in the stochastic case, taking expectations leads, in view of the properties of the Ito calculus, to terms that are not possible to handle by the available estimates. We refer to \cite{LPSscl} for an extended discussion of this point. 

\smallskip

\noindent  Our interest in sscl is twofold. Given the theory of stochastic viscosity solutions and the connection between conservation laws and Hamilton-Jacobi equations when $N=1$, it is very natural from the mathematical point of view to ask whether there is such a theory for the former. The second reason is that sccl like \eqref{xscl} arise naturally as models in the Lasry-Lions theory of mean field games (\cite{ll1}, \cite{ll2}, \cite{ll3}). We  also refer to \cite{LPSscl} for an example of stochastic system of interacting particles. 

\smallskip

\noindent  If, instead of \eqref{path}, we assume that ${W} \in C^1((0,\infty);\R^N)$, then \eqref{xscl} is a ``classical'' problem with a well known theory; see, for example, the books of Dafermos \cite{DafBook} and Serre \cite{SerreBook}. The solution can develop singularities in the form of shocks (discontinuities). Hence it is necessary to consider entropy solutions which, although not regular, satisfy the $L^1$-contraction property established by Kruzkov that yields  uniqueness.  Entropy solutions, which are based on certain inequalities, cannot be used when {W} is not smooth.

\smallskip

\noindent There are several challenges when trying to extend the analysis of  \cite{LPSscl} to spatially dependent fluxes. The underlying  idea is the same, namely  to use the  characteristics at the kinetic level to ``eliminate the bad stochastic terms''. In the spatially independent setting the characteristics can be solved explicitly and, as a result, it is possible to keep track of all the cancellations that are taking place and, hence, obtain rather strong estimates.  In the inhomogeneous 
setting there are no explicit solution  of the system of the stochastic characteristics. As a result the calculations are not transparent and it becomes necessary to  employ more complicated arguments to study the problem.

\smallskip

\noindent We remark that recently Debussche and Vovelle \cite{dv} (see also Feng and Nualart \cite{fn}, Chen, Ding and Karlsen \cite{cdk}, Debussche, Hofmanov{\'a} and Vovelle \cite{DHV}, Hofmanov{\'a} \cite{Hof1}, \cite{Hof2}, Berthelin and Vovelle \cite{BerVov} and Debussche and Vovelle \cite{DebVov100}) put forward a theory of weak entropy solutions of scalar conservation laws with Ito-type semilinear  stochastic dependence. Such problems do not appear to be amenable to a pathwise theory. Our results do not cover the equations studied in \cite{dv, fn} and vice versa. We refer to Section 6 of \cite{LPSscl} for a discussion of these issues.

\subsection*{Organization of the paper.}
The paper is organized as follows: In Section~\ref{sec:kf} we state the assumptions, we review briefly some facts about entropy solutions and the kinetic formulation  and we introduce the notion of the pathwise solutions. In Section~\ref{technical} we summarize  some of the tools we use and prove a preliminary result. 
Section \ref{sec:proof} and Section \ref{sec:ex} are  devoted respectively to the proofs of the contraction  and intrinsic uniqueness of the stochastic entropy solutions and their existence. In Section \ref{sec:ex} we also state and prove general $L^\infty$- and $L^2$-bounds.

\section{The kinetic formulation and pathwise stochastic entropy solutions}
\label{sec:kf}
\noindent  We consider the sscl \eqref{xscl}. We write
$$ 
\dv {\bf A}(x,u) = \dv_x {\bf A}(x,u) + \dv_u {\bf A}(x,u),
$$ 
where
$$
\dv_x {\bf A}(x,u)=\sum_{i=1}^N A_{i,x_i} (x,u) \ \text{ and } \ \dv_u {\bf  A}(x,u)=\sum_{i=1}^N A_{i,u} (x,u)u_{x_i}, 
$$
and introduce the notation
\beq\label{a}
{\bf a}(x,\xi)=(a_1(x, \xi), \dots, a_N(x, \xi)):={\bf A}_u(x, \xi)= (A_{1,u}(x, \xi), \ldots, A_{N,u}(x, \xi))\\[2mm]   
\eeq
and
\beq\label{b}
b(x, \xi) = \dv_x {\bf A} (x, \xi).
\eeq
\subsection*{Assumptions.} We summarize  here the main assumptions we need in the paper.  In addition to \eqref{flux} and \eqref{path}, we assume 
that 
\beq
b(x,0) = 0,
\label{as:b}
\eeq
and
\beq\label{assumptions1}
\p_j a_i, \;  \p_u a_i,  \;  \p_j b,  \;  \p_u b \in L^\infty (\R^N \times \R)  \ \  \text{for} \ \ 1\leq i,j  \leq N.
\eeq

\subsection*{Entropy solutions for smooth paths and some estimates.}  If $W \in C^1([0,\infty);\R)$ and $u,v$ solve \eqref{xscl}, the classical entropy inequalities yield
$$
d |u-v| +  \dv \big[ [{\bf{A}}(x, u)- {\bf{A}}(x, v)] \; \sgn(u-v) \big] \circ dW \leq  0,
$$
and thus
\beq\label{1}
\f {d}{dt} \int_{\R^N}  |(u-v)(x,t)| dx \leq 0.
\eeq
Since, in view of \eqref{as:b},  $v\equiv 0$ is a solution, the contraction yields the a priori $L^1$-estimate \beq\label{l1}
\| u (\cdot, t) \|_{1} \leq \| u^0 \|_{1} .
\eeq
On the other hand $L^\infty$-estimates, which are also known as invariant regions,  of the form 
\beq\label{linfinity}
 \| u(\cdot, t) \|_{\infty} \leq C(t)
\eeq
for space dependent fluxes are not as general and require some additional work and stronger assumptions.  We give a general argument in Section~\ref{sec:ex} and discuss some examples that  are not covered by it. 

\subsection*{The kinetic formulation for smooth paths.}  We review here the basic concepts of the kinetic theory of scalar conservation laws and we show that it allows to define a change of variable along the ``kinetic'' characteristics which turns out to be a very convenient tool for the study of sscl. 

\smallskip

\noindent  Although we use the notation of the introduction, here we assume that
$W \in C^1([0,\infty);\R),$
in which case $du$ stands for the usual derivative and $\circ$ is the usual multiplication and, hence, should be ignored.

\smallskip

\noindent  The entropy inequalities  (see, for example, \cite{DafBook, SerreBook}), which yield \eqref{1} and guarantee the uniqueness of the entropy solutions, are 
\beq
\left\{\begin{array}{l}
dS(u) + \displaystyle \left[\dv_u ({\bf A}^S(x,u))
+ S'(u)  [\dv_x {\bf A}](x,u) 
- B^S(x,u) \right]  \circ dW \leq 0  \ \text{ in } \ \R^N\times(0,\infty),
\\[5mm]
S( u)=S(u^0) \quad \text{ on } \  \R^N\times\{0\},
\end{array} \right.
\label{eq:rsentr}
\eeq
for all $C^2-$convex functions $S$, entropy fluxes ${\bf A^S}$  and forcing terms ${B^S}$ defined by
$$
{\bf A^S}(x,u)_u = {\bf A}_u(x, u) S'(u)  \quad  \text{ and} \quad  {B^S}(x,u)_u = \dv {\bf A}_u(x, u) S'(u) .
$$

\smallskip

\noindent  It is by now well established that the simplest way to study conservation laws is to use  the kinetic formulation developed in a series of papers -- see Perthame and Tadmor \cite{PTscal}, Lions, Perthame and Tadmor \cite{LPTscal}, Perthame \cite{Peuniq, PeKF}, and Lions, Perthame and Souganidis \cite{LPS1,  LPS2}. The kinetic formulation for inhomogeneous in space but time independent conservation laws was developed by Dalibard \cite{ALD}. The arguments of \cite{ALD} easily extend to problems with smooth multiplicative paths. 
\smallskip

\noindent The idea is to write a linear equation for the nonlinear function $\chi:\R^N \times\R\times(0,\infty) \to \R$ given by 
\beq
\chi(x,\xi, t): = \chi (u(x,t), \xi)=   \left\{\begin{array}{l}
+1 \quad \text{ if } \quad 0 \leq \xi \leq u(x,t),
\\[2mm]
-1  \quad \text{ if } \quad u(x,t) \leq \xi \leq 0,
\\[2mm]
\ \;  0 \qquad \text{ otherwise}.
\end{array} \right.
\label{eq:chi}
\eeq

\smallskip

\noindent  The kinetic formulation says that, if $u^0 \in (L^1\cap L^\infty)(\R^N)$,  the set of entropy inequalities \eqref{eq:rsentr} for all convex entropies $S$ can be replaced by the equation 
\beq \left\{\begin{array}{l}
d \chi + \displaystyle  \dv_x [{\bf a} (x, \xi)\chi ] \circ dW -  \dv_\xi [b (x, \xi)\chi ] \circ dW   = \partial_\xi m dt   \  \text{ in } \ \R^N \times \R \times (0,\infty),
\\[3mm]
\chi =\chi(u^0(\cdot), \cdot) \ \text{ on } \quad \R^N\times\R\times\{0\},
\end{array} \right.
\label{eq:skf}
\eeq
which is supposed to be satisfied in the sense of distributions, where 
\beq\label{measure0}
m \ \text{ is a nonnegative  measure in } \  \R^N \times \R \times (0,\infty)
\eeq
and
\beq\label{measure100}
\begin{cases}
\int_0^T \int_{\R} \int_{\R^N} m(x,\xi,t)dx d\xi dt  \leq \frac{1}{2} \|u^0\|_{2} + \int_0^T \int_\R \int_{\R^N} |b(x,\xi)||\dot W| \chi d\xi dx dt  \  \text{for all $T>0$ \   and} \\[2mm] 
m(x,\xi,t) =0 \  \text{ for } \  |\xi| \geq \|u(\cdot, t)\|_\infty;
\end{cases}
\eeq
note that the estimate on mass is new even in the context of the ``classical'' conservation laws, that is without rough path dependence.
\smallskip

\noindent  We remark that although to prove the equivalence between entropy and kinetic solutions it is necessary to work in $(L^1\cap L^\infty)(\R^N)$, the kinetic theory of \cite{ALD} is set only in $L^1.$ Hence  the condition on the support of the measure, which is equivalent to an $L^\infty$-bound for  $u$, 
is not needed provided it is assumed that $b\in L^1(\R^{N}\times \R).$  
More importantly we note that a priori bound for the total mass of the measure in \eqref{measure100} is useless for irregular paths since it depends on the total variation  of the path.
\smallskip

\noindent Using some of the ideas in this paper it is, however, possible to improve the upper bound so that it becomes independent of the variation of the path. In Section \ref{sec:ex} we sketch the proof of such an estimate as well as an $L^\infty$-bound on the solutions.
We summarize these in
\beq\label{measure}
\begin{cases}
\text{ 
there exists $C_T >0$ such that }  \int_0^T \int_{\R} \int_{\R^N} m(x,\xi,t)dx d\xi dt  \leq C_T \|u^0\|^2_{2}, 
\\[1.5mm]  
\text{and } \ \\[1.5mm]
m(x,\xi,t) =0 \  \text{ for } \  |\xi| \geq \|u(\cdot, t)\|_\infty.\end{cases}
\eeq

\smallskip

\noindent  Notice that the transport equation for $\chi$ has a Hamiltonian type structure and is equivalently written in the strong form 
\beq
d \chi + \displaystyle {\bf a} (x, \xi)\cdot D_x \chi  \circ dW -  b (x, \xi) D_\xi \chi  \circ dW   =\partial_\xi m dt  .
\label{strongform}
\eeq

\smallskip

\noindent  For $\rho^0 \in C^\infty_{\rm b} ( \R^N \times \R)$ and $t_0\geq 0$, let $\rho(\cdot, \cdot ,\cdot; t_0)$ be the solution of the linear stochastic pde
\beq \label{rho}
\begin{cases}
d \rho + [ {\bf a}(x,\xi)\cdot D_x \rho - b(x,\xi)D_\xi \rho] \circ dW=0  \      \text{ in } \  \R^N \times \R \times (t_0,\infty), 
\\[2mm]
\rho = \rho^0  \  \text{on}  \ \R^N \times \R \times \{t_0\}.
\end{cases}
\eeq

\smallskip

\noindent  Since $W$ is a single path, it is immediate that 
$$
\rho (x,\xi,t;t_0)= \hat \rho (x,\xi,W(t)-W(t_0)),
$$
where, in view of the assumptions above, $\hat \rho$ is the smooth solution of 
\beq \label{hatrho}
\begin{cases}
\hat \rho_t + {\bf a}(x,\xi)\cdot D_x \hat \rho - b(x,\xi)D_\xi \hat \rho=0  \      \text{ in } \  \R^N \times \R \times \R, 
\\[2mm]
\hat \rho = \rho^0   \  \text{on}  \ \R^N \times \R \times \{0\}. 
\end{cases}
\eeq

\smallskip

\noindent  Next for  $\rho^0 \in C^\infty_{\rm b} (\R^N \times \R)$, $t_0 \geq 0$ and    $y, \eta \in \R^N$ we consider the solution 
$\rho(\cdot, \cdot, y, \eta, \cdot; t_0)$ of \eqref{rho} with initial datum, at $t=t_0$,  $\rho^0(\cdot - y, \cdot -\eta)$ and introduce
the ``convolution''  along characteristics given by 
\beq
 \rho \star \chi (y, \eta, t; t_0) := \int  \rho (x,\xi, y, \eta, t; t_0)  \chi(x, \xi,t) dx d\xi . 
\label{char_convol}
\eeq

\smallskip

\noindent  Combining  \eqref{strongform} and  \eqref{rho}  we find  that, in the sense of distributions,
\beq
\f{d}{dt} \rho \star \chi (y, \eta, t; t_0) = - \int \p_\xi  \rho (x,\xi, y, \eta , t; t_0) m(x,\xi,t) dx d\xi.
\label{strongform_reg}
\eeq
It follows that, for all  almost every  $t_0 \geq 0$ and almost all $t\geq 0,$
\beq\label{strongform1}
 \rho \star \chi (y, \eta, t) =  \rho \star \chi^0 (y, \eta,t_0; t_0)- \int_{t_0}^{t} \int \p_\xi  \rho (x,\xi, y, \eta , t) m(x,\xi,s) dx d\xi ds .
\eeq
The $L^1$-continuity of $u$ at $t=0$, as stated later, implies that this formula holds for $t_0=0$ and for all  almost every $t\geq 0$.

\smallskip

\noindent  Note that although the regularity of the path was used to derive \eqref{strongform_reg} and \eqref{strongform1}, the  actual statements do not need it and  make sense for paths which are only continuous.
Notice also that \eqref{strongform_reg} and   \eqref{strongform1} are equivalent to the kinetic formulation when the measure $m$ satisfies \eqref{measure0} and \eqref{measure}.

\subsection*{Pathwise stochastic entropy solutions.} 
Neither the notion of entropy solution nor the kinetic formulation can be used to study \eqref{xscl}, since both involve either entropy inequalities or the sign of the defect measure quantities which do not make sense for equations/expressions with, in principle, are nowhere differentiable functions. We refer to \cite{LSsemilinear, LSuniqueness, LShjscras98,LShjscras98b} for a general discussion about the difficulties encountered when attempting to use the classical weak solution approaches to study fully nonlinear stochastic pde.

\smallskip

\noindent Motivated by our previous work \cite{LPSscl} as well as the theory of stochastic viscosity solutions (\cite{LSsemilinear, LSuniqueness, LShjscras98,LShjscras98b}) we use  \eqref{strongform_reg} to introduce next the notion of pathwise stochastic entropy solutions for SSCL. 
Recall, as remarked previously, that the statement of  \eqref{strongform_reg} does not rely on any regularity of the paths.

\smallskip

\noindent We have:
\begin{definition} \label{def:xsolution} 
Assume \eqref{flux}, \eqref{path},  \eqref{as:b} and \eqref{assumptions1}. Then $u\in (L^1\cap L^\infty)(\R^N \times (0,T))$, for all $T>0$,  is a pathwise stochastic entropy solution of \eqref{xscl}, if there exists a nonnegative bounded measure $m$ on $\R^N \times \R \times (0,\infty)$ satisfying \eqref{measure} such that \eqref{strongform_reg} and \eqref{strongform1} hold for all $\rho$'s given by \eqref{char_convol} with $\rho^0 \in C^\infty_{\rm b}(\R^N\times \R)$,  almost all $t_0 \geq 0$ and almost everywhere in $t\in [t_0,T].$
\end{definition}
\smallskip

\noindent  We show in the next two sections that the pathwise stochastic entropy solutions exist and satisfy a contraction in $L^1$ and, hence, are stable and unique.

\section{Some technical results}
\label{technical}
%
\noindent We discuss here the class of  test functions we use in the paper, recall the method of characteristics which provide a representation for the solutions to \eqref{rho} and \eqref{hatrho} and show what they imply for the special solutions we are considering, and, finally, state and prove a technical result. 

\subsection*{The test functions}  For each $\e>0$ we consider test functions solving solving \eqref{strongform} with initial data $\rho_\e^0$ that separates space, $x\in R^N$, and velocity, $\xi \in \R$.  A natural  choice is  convolution approximation to Dirac masses, that is 
\beq \label{init:conv}
\rho^0_\e (x-y, \xi - \eta)= \rho_\e^s  (x-y) \rho_\e^v ( \xi -\eta ), 
\eeq
where, as $\e \to 0$ and in the sense of distributions, 
\beq
 \rho_\e^s
 {\longrightarrow}  \delta(x-y)   \  \text{and}  \   \rho_\e^v  {\longrightarrow}   \delta(\xi -\eta); 
\eeq
here we use the superscripts $s$ and $v$ to signify whether the initial data approximates the space or the velocity variables. Moreover $\delta_a$ is the Dirac mass at $a$ and we write $\delta$ if $a=0$. Typical choices for $\rho_\e^s$ and $\rho_\e^v$ are 
$$\rho_\e^s(x)=\e^{-N}\rho^s(x/\e) \ \text{and } \  \rho_\e^v(\xi)=\e^{-1}\rho^v(\xi/\e),$$ 
for some smooth functions $\rho^s$ and  $\rho^v$  with compact support of diameter $1$ and such that 
$$0\leq \rho^s, \rho^v \leq 1, \  \rho^s(0)=\rho^v(0)=1  \ \text{and}  \int_{\R^N} \rho^s= \int_{\R} \rho^v =1.$$

\smallskip

\subsection*{The characteristics} Depending on the context we use the forward and backwards characteristics of \eqref{hatrho} to construct the smoothing kernels $\rho_\e$ and $\hat\rho_\e$, that is the solutions 
to \eqref{rho} and \eqref{hatrho} respectively. 
 
\smallskip

\noindent  Using   the forward characteristics
\beq \begin{cases}
\dot Y_{(y,\eta)} (s) =  {\bf a} (Y_{(y,\eta)}(s), \zeta_{(y,\eta)} (s)) , &\quad Y_{(y,\eta)}(0)= y,
\\[4mm]
\dot \zeta_{(y,\eta)} (s) = -b(Y_{(y,\eta)} (s),\zeta_{(y,\eta)} (s) ) ,         &\quad \zeta_{(y,\eta)} (0)= \eta, 
\end{cases}
\label{characF}
\eeq
we generate the  smoothing kernels $\wh \rho_\e$ and $\rho_\e$ 
and 
which satisfy, in the sense of distributions,
\beq
\wh \rho_\e(x,\xi, y, \eta,s)   \underset{\e \to 0}{\longrightarrow}  \delta (x-Y_{(y,\eta)}(s)) \delta (\xi- \zeta_{(y,\eta)}(s))
\label{x:limrhohat}
\eeq
and 
\beq
 \rho_\e(x,\xi, y, \eta,t)   \underset{\e \to 0}{\longrightarrow} \;  \delta (x-Y_{(y,\eta)}(W(t)-W(t_0))) \delta (\xi- \zeta_{(y,\eta)}(W(t)-W(t_0))).
\label{x:limrho}
\eeq

\smallskip

\noindent  We may also use the backward characteristics, that is the solution of 
\beq \begin{cases}
\dot X_{(t,x,\xi)} (s) = \  {\bf a} (X_{(t,x,\xi)}, \Xi_{(t,x,\xi)} ) , &\quad X_{(t,x,\xi)}(t)= x,
\\[4mm]
\dot \Xi_{(t,x,\xi)} (s) = -b(X_{(t,x,\xi)}, \Xi_{(s,x,\xi)}) ,         &\quad \Xi_{(t,x,\xi)}(t)= \xi;
\end{cases}
\label{characBack}
\eeq
then
$$
\wh \rho_\e(x,\xi,y, \eta , t) = \rho_\e^0 \big(X_{(t,x,\xi)}(0)-y, \Xi_{(t,x,\xi)} (0)-\eta \big)
$$
and
$$
\rho_\e(x,\xi,y, \eta , t; t_0) = \rho_\e^0 \big(X_{(W(t)-W(t_0),x,\xi)}(0)-y, \Xi_{(W(t)-W(t_0),x,\xi)} (0)-\eta \big)
$$

\noindent We note  that, in view of \eqref{as:b} and the uniqueness of the solutions to \eqref{characF} and \eqref{characBack}, 
\beq\label{takis}
 \zeta_{(y,0)}(t)=\Xi_{(t,x,0)}=0, \    \sgn( \zeta_{(y,\eta)}(t))=\sgn ( \eta )  \ \text { and } \  \sgn({ \Xi_{(t,x,\xi)}}(s))=\sgn (\xi ) .
\eeq

\noindent  To keep the notation simple in the sequel we write sometimes $(X_{(x,\xi)}, \Xi_{(x,\xi)})$ for \\
$(X_{(W(t)-W(t_0),x,\xi)}(0), \Xi_{(W(t)-W(t_0),x,\xi)} (0))$ and
$$
 \rho_\e(x,\xi, y, \eta,t)  =  \rho_\e^s  (X_{(x,\xi)}-y) \rho_\e^v ( \Xi_{(x,\xi)} -\eta ).
$$

\subsection*{Some  technical facts.}  We present now two technical facts which we will use in the next section. It concerns the behavior  of  
\beq\label{barq}
\bar q_\e (x,\xi, t;t_0):= \int_{\R^N} \int_\R q_\e(x,\xi,y,\eta, t;t_0)dy d\eta \ \text{ and } \ \p_\xi\bar q_\e (x,\xi, t;t_0),
\eeq
where, for each $y \in \R^N$ and $\eta \in \R$, $q_\e(x,\xi,y,\eta, t;t_0)$ is the solution of \eqref{rho} starting 
with initial datum
\beq\label{qq}
q_\e(x,\xi, y, \eta, t_0;t_0):= \rho_\e^s(x-y)  \rho_\e^v(-\eta) (- \frac{1}{2}  +  \int_{-\infty}^\xi \rho_\e^v(\bar \xi -\eta) d\bar\xi);
\eeq
it is, of course, immediate that $\bar q_(\cdot,\cdot,\cdot;t_0)$ is the solution of \eqref{rho} with initial datum
$$\bar q_\e(x,\xi, t_0; t_0)= - \frac{1}{2} + \int_\R \int_{-\infty}^\xi \rho_\e^v(\bar \xi -\eta) \rho_\e^v(-\eta) d\bar\xi d\eta. $$
\noindent The result is:
\begin{lemma}\label{delta100}
Assume \eqref{flux}, \eqref{path}, \eqref{as:b} and \eqref{assumptions1} and, for $t_0 \geq 0$ and $\e>0$, let  $\bar q_\e(\cdot, \cdot, \cdot;t_0)$  be given by \eqref{barq}. 
As $\e \to 0$,  for all times $t$ and a.e. in $(x,\xi),  \ \bar q_\e(x,\xi,t;t_0)  \to \f 12  {\rm sign} (\xi) $ and, thus, in $L_{\text{loc}}^p(\R^N\times \R)$ for all $p\in [1,\infty).$ \end{lemma}
\begin{proof}
In view of the relationship between the solutions to \eqref{rho} and \eqref{hatrho}, it suffices to prove the claim for the solution
$\hat q_\e$ of \eqref{hatrho}  with $\hat q_\e (x,\xi, y,\eta, 0)$ as in \eqref{qq}.
\smallskip

\noindent 
In light of the discussion in the previous paragraph, for each fixed $t$,  we have 
$$ \begin{array}{rl}
\dis \int_{\R^{N}}\int_{\R} \hat q_\e (x,\xi, y,\eta, t)  dy d \eta & = \dis \int_{\R^{N}}\int_{\R} \hat q_\e (X_{(t,x,\xi)} (0), \Xi_{(t,x,\xi)} (0), y, \eta, 0) dy d \eta
\\ \\
&= - \frac{1}{2} +   \dis \int_{\R} 
 \int_{-\infty}^{ \Xi_{(t,x,\xi)} (0)}  \rho_\e^v(\bar \xi -\eta) \rho_\e^v(-\eta) d\bar\xi d\eta
\\ \\
& \underset{ \e \to 0 }{\longrightarrow} - \frac{1}{2} + \ind{\Xi_{(t,x,\xi)} (0) >0} = - \frac{1}{2} + \ind{\xi >0},
\end{array}
$$
the last equality being  a consequence of  \eqref{takis}. 
\end{proof}

\smallskip

\noindent As far as $\p_\xi\bar q_\e (x,\xi, t;t_0)$  is concerned we first observe that
$$\begin{array}{rl}
\p_\xi \hat q_\e(x,\xi,y,\eta, t;t_0)= &  -\rho_\e^v(-\eta)  D_y \rho_\e^s(X_{(t,x,\xi)}(0)-y) \big[ - \frac{1}{2}  +  \int_{-\infty}^{\Xi_{(t,x,\xi)}(0)}\rho_\e^v(\bar \xi -\eta) d\bar\xi \, \big] \p_\xi X_{(t,x,\xi)}(0)
\\[5pt]
& + \rho_\e^v(-\eta)  \rho_\e^s(X_{(t,x,\xi)}(0)-y) \rho_\e^v(\Xi_{(t,x,\xi)}(0)-\eta) \p_\xi \Xi_{(t,x,\xi)}(0).
\end{array}$$ 

\noindent Integrating in $y$ and $\eta$ we find
\beq\label{takis1}
\p_\xi\bar q_\e (x,\xi, t;t_0)= \int_{\R^N}\int_\R  \rho_\e^v(-\eta) \rho_\e(x,\xi,y,\eta,t;t_0) d\eta dy {\p_\xi\Xi_{(W(t)-W(t_0),x,\xi)}(0)}.
\eeq 

\section{The stability and uniqueness of pathwise stochastic entropy solutions} 
\label{sec:proof}

\subsection*{The result} The  result  about the $L^1$-contraction property and, hence, the intrinsic uniqueness of the pathwise stochastic entropy solutions, is stated next.

\begin{theorem} \label{th:xsolution}Assume \eqref{flux}, \eqref{path}, \eqref{as:b}, \eqref{assumptions1} and $u^0  \in (L^1\cap L^\infty)(\R^N)$ and fix $T>0$.  There exists at most one  pathwise stochastic entropy solution $u \in L^\infty \big( (0,T); L^1 \cap L^\infty(\R^N) \big)$  to \eqref{xscl} which is continuous at $t=0$ with values in $L^1(\R^N).$
In addition any two pathwise stochastic entropy solutions  $u_1, u_2$ 
satisfy, for almost all $t>0$, the ``contraction'' property
\beq\label{cont}
\| u_2(\cdot,t) -u_1(\cdot,t) \|_{1} \leq \| u_2^0-u_1^0 \|_{1}.
\eeq
\end{theorem}

\subsection*{The proof} When $W$ is smooth, the proof of \eqref{cont} is based on 
considering, for two solutions $u_1$ and $u_2$ and $\chi^{(i)} ( x,\xi,t)=\chi(u_i( x,t); \xi)$,  the function
$$
F(t) :=  Ê\int \big[ \; |\chi^{(1)} ( x,\xi,t)| + | \chi^{(2)} ( x,\xi,t)| -2 \chi^{(1)} ( x,\xi,t)\chi^{(2)} ( x,\xi,t) \big] dx d\xi = \| u_1 (t)- u_2 (t) \|_1
$$
and  showing that $dF/dt \leq 0$, which yields the contraction property 
$$
F(t_2) \leq F(t_1)  \ \text{for almost all } \ 0\leq t_1 \leq t_2.
$$

\noindent This is, however, not possible for the paths we are considering in this note. It is, therefore, necessary  to modify the integrand in $F(t)$  in order to make use of the defining properties of stochastic entropy solutions. In particular we need to replace $ |\chi(u_i(x,t),\xi)|=\text{sign}(\xi)\;   \chi(u_i(x,t),\xi)$ by the ``convolution'' along characteristics \eqref{char_convol} for an appropriate choice of $\rho$, which is going to be the $q_\e(\cdot,\cdot,\cdot;t_0)$ in  \eqref{barq}.
We also need to localize in time and use an iteration. The reason is that the approximation creates errors which can be controlled by the oscillations of the path $W.$ Thus we need to 
discretize in time and, hence, to use $\rho\e$'s and  $q_\e$'s starting at different times and to add the errors. 

\noindent We present now the

\begin{proof}[The proof of Theorem \ref {th:xsolution}] The uniqueness follows immediately from the contraction property, hence, here we concentrate on the latter. Since the  proof is long and technical we divide it in several steps some of which, although formally true, require justification. 
\smallskip

\noindent {\it The general set up.} Fix $T>0$. We begin with a regularization of the functional $F$. For  $t_0\geq 0$, $\e>0$ small,  for $t\geq t_0$ and $q_\e$ as in Lemma \ref{delta100}, let
$$ \begin{array}{rl}
F_{t_0,\e}(t) := Ê\dis \int \big[& q_\e (\cdot,t;t_0)\star \chi^{(1)} ( y, \eta,t) + q_\e (\cdot,t;t_0)\star \chi^{(2)} ( y, \eta,t)
\\[10pt]
& - 2  \rho_\e(\cdot,t;t_0) \star\chi^{(1)} ( y, \eta,t)\;  \rho_\e(\cdot,t;t_0) \star \chi^{(2)} ( y, \eta,t) \big] dy d \eta.
\end{array}$$

\noindent It follows from the a.e. continuity in time of the $F$ as well as  \eqref{x:limrho} 
that,  for almost every  $t \in [0,T]$ as $\e\to 0$,
\beq\label{app}
F_{t_0,\e}(t) \rightarrow F(t).
\eeq 
\noindent  For $h>0$, let $\omega (h)$ denote the oscillation of the path $W$ over time intervals of size $h$, that is 
\beq
  \omega(h ): = \sup_{0 \leq s \leq h, \ 0\leq t \leq T }  |W(t+s) - W(t)|.
\label{eq:omega}
\eeq

\noindent Fix $s,t$ so that $0\leq s\leq t\leq T$ and let $\Delta=\{s=t_0 \leq t_1 \ \leq \ldots \leq t_M=t\}$ be a partition of $[s,t]$ with mesh size $ h= t_{i+1}-t_i$ and such that \eqref{app} holds for all $t_i  \in \Delta$.  The conclusion follows if we show that there exists $C>0$ such that 
\beq\label{enough}
F_{t_i,\e}(t_{i+1}) - F_{t_i,\e}(t_i) \leq  
C \omega(h) \int_{t_i}^{t_{i+1}} \int_{\R^N} \int_\R [m^{(1)}(x,\xi,t) + m^{(2)}(x,\xi,t)]d\xi dx dt.
\eeq

\smallskip

\noindent Indeed if \eqref{enough} holds, then 
$$\begin{array}{rl} 
F(t) - F(s) & =  \dis \sum_{i=0}^{M-1}  [F(t_{i+1} -F(t_i)]= \lim_{\e\to 0} \dis \sum_{i=0}^{M-1} [F_{t_i,\e}(t_{i+1}) - F_{t_i,\e}(t_i)]
\\[10pt]
& \leq C \omega (h) \dis \int_{s}^{t} \dis \int_{\R^N} \int_\R [m^{(1)}(x,\xi,t) + m^{(2)}(x,\xi,t)]d\xi dx dt.
\end{array}$$ 
\noindent In view of \eqref{measure}, to prove \eqref{enough}  it suffices to show that,  for almost every  $t_0 \geq 0$ and  $h, \e >0$, there exists $C>0$ such that almost everywhere  in  $(t_0, t_0 +h)$ 
$$
\frac{dF_{t_0,\e}} {d t} \leq C\omega (h)  \int_{\R^N} \int_\R m(x,\xi,t) d\xi dx ;
$$ 
notice that this can be written only for  continuity times of the right hand side, which include $t_0=0$ and this is enough for our purpose.
\smallskip 

\noindent  Since the upper bound for  $dF_{t_0,\e}/dt $ requires some long and tedious calculations, we divide the computations and estimates in several parts. In the first, which is the longest, we estimate  the derivative of the product term. The second is about about the single terms. After grouping everything together we find the error term. In what follows, to keep the notation simple,  we assume that $t_0=0$, write $F_\e$ instead of $F_{0,\e}$, we omit the dependence of the solutions of \eqref{rho} and \eqref{hatrho} on the initial time and we write $X_{(x,\xi)}, \ \Xi_{(x,\xi)},  \ X_{(x',\xi')}$ and  $\Xi_{(x',\xi')}$ in place of $X_{(W(t),x,\xi)}(0), \ \Xi_{(W(t), x,\xi)}(0),  \ X_{(W(t), x',\xi')}(0)$ and  $\Xi_{(W(t), x',\xi')}(0).$
\smallskip

\noindent {\it The product term in $\f{d}{dt} F_\e$.}  It follows from the definitions that   \beq\label{product11}
- \f{d}{dt} [ \rho_\e \star \chi^{(1)} ( y, \eta,t)  \rho_\e \star \chi^{(2)} ( y, \eta,t)] = I_1(t) + I_2(t),
\eeq
where 
\beq\label{prod1}
\begin{cases}
I_1(t) := \int \p_\xi  \rho_\e (x,\xi, y, \eta , t) m^{(1)}(x,\xi,t) \rho_\e (x',\xi' , y, \eta , t) \chi^{(2)}(x',\xi',t) dx d\xi dx'd\xi', \\[2.5mm]
\text{and}\\[1.5mm]
I_2(t):= \int \rho_\e (x,\xi, y, \eta , t) \chi^{(1)}(x,\xi,t)  \p_\xi  \rho_\e (x',\xi', y, \eta , t) m^{(2)}(x',\xi',t) dx d\xi dx'd\xi'  .
\end{cases}
\eeq

\noindent  We focus next on $I_1(t)$ since $I_2(t)$ 
is handled similarly. 
\smallskip

\noindent  The classical proof of uniqueness for smooth paths uses $\rho^s(x-y)\rho^s(x'-y)\rho^v(\xi -\eta)\rho^v(\xi'-\eta)$ in place  of  $\rho(x,\xi,y,\eta,t) \rho_\e (x',\xi' , y, \eta , t).$ As a result,  instead  of $\p _{\xi} \rho(x,\xi,y,\eta,t) \rho_\e (x',\xi' , y, \eta , t) $,  the integrand has    $\p_{\xi} \rho_\e^v(\xi - \eta) \rho_\e^v(\xi' - \eta )$,  and, hence, it is easy to interchange $\xi$ and $\xi'$ derivatives to obtain cancellations. This is, however, not the case here and we need to introduce the appropriate derivatives, a fact that gives rise to additional terms and error terms that need to be approximated.  
\smallskip

\noindent  After integrating with respect to  $y$ and $\eta$ and integrating by parts, we find 
$$
I_1(t)=I_{11}(t) + I_{12}(t),
$$
with
$$
I_{11}(t):= - \int  m^{(1)} (x,\xi,t)   \chi^{(2)} (x',\xi',t) \rho_\e(x,\xi , y , \eta,t)  {\p_ {\xi'}} \rho_\e(x',\xi',y, \eta,t) dy  d\eta dx d\xi dx' d\xi',
$$
and
$$\begin{array}{rl}
I_{12}(t):=  \dis \int  m^{(1)}(x,\xi,y,\eta,t) \chi^{(2)}(x',\xi',y,\eta,t) & \big[   \p_\xi \rho_\e(x,\xi ,y , \eta,t)  \rho_\e(x',\xi',y, \eta,t)
\\[5pt]
& +  \rho_\e(x,\xi ,y , \eta,t)  \p_{\xi'} \rho_\e(x',\xi',y, \eta) \big] dy  d\eta dx d\xi dx' d\xi'.
\end{array}$$
\smallskip

\noindent  Integrating by parts in $I_{11}(t)$ and using that $\p_{\xi'} \chi^{(2)}(x',\xi',t)=\delta (\xi') - \delta_{u^2(x',t)}(\xi') \leq \delta (\xi')$, which is an immediate
consequence of the definition of $\chi^{(2)}$ and the positivity, in the sense of distributions, of the Dirac mass, we find 
$$ 
I_{11}(t) = \int  m^{(1)} (x,\xi,t) \p _{\xi'}  \chi^{(2)} (x',\xi',t) \rho_\e(x,\xi , y , \eta,t) \rho_\e(x',\xi',y, \eta,t) dy  d\eta dx d\xi dx' d\xi' \leq I_{13}(t)$$

\noindent with
$$ I_{13}(t):= \int  m^{(1)} (x,\xi,t) \rho_\e(x,\xi , y , \eta,t)  [ \int \rho_\e(x',\xi',y, \eta,t) \delta(\xi') d\xi' dx'] dy d\eta dx d\xi.$$

\noindent Next we use that,  in light of $\delta(\xi')= \f 12 \p_{\xi'} \sgn (\xi')$,
 
$$\int \rho_\e(x',\xi',y, \eta,t) \delta(\xi') d\xi'dx' =  - \f 1 2 \int  \p_{\xi'} \rho_\e(x',\xi',y, \eta,t) \sgn (\xi') d\xi'dx,$$
\smallskip

\noindent  and rewrite $I_{13}(t)$ as 
$$I_{13}(t) = I_{14}(t) + I_{15}(t),$$

\noindent  with 

$$ I_{14}(t) := \f 12  \int  m^{(1)} (x,\xi,t)  \p_\xi \rho_\e(x,\xi , y , \eta,t)[ \int  \rho_\e(x',\xi',y, \eta,t) \sgn (\xi') d\xi'dx'] dy d\eta dx d\xi, $$

\noindent and
$$\begin{array}{rl}
I_{15}(t) :=  -  \f 12 \dis \int m^{(1)} (x,\xi,t) \sgn (\xi') \big[& \rho_\e(x',\xi',y, \eta,t) \p_\xi \rho_\e(x,\xi , y , \eta,t)
\\[5pt]
&+ \rho_\e(x,\xi , y , \eta,t) \p_{\xi'} \rho_\e(x',\xi',y, \eta,t) \big] dx' d\xi ' dy d\eta dx d\xi.
\end{array}$$
\smallskip

\noindent Since  $\rho_\e (x,\xi,y,\eta,t)= \rho_\e^s(X_{(x,\xi)} -y)   \rho_\e^v(\Xi_{(x,\xi)} -\eta),$  
we find

$$
\p_\xi \rho_\e(x,\xi , y , \eta,t) = -D_y \rho_\e^s(X_{(x,\xi)} -y)   \rho_\e^v(\Xi_{(x,\xi)} -\eta) \p_\xi X_{(x,\xi)} -  \rho_\e^s(X_{(x,\xi)} -y)  D_\eta \rho_\e^v(\Xi_{(x,\xi)} -\eta) \p_\xi \Xi_{(x,\xi)}.
$$
\smallskip

\noindent We continue with two key observations. In light of  \eqref{takis} and the facts  that
$$
\rho_\e (x',\xi',y,\eta,t)= \rho_\e^s(X_{(x',\xi')} -y) \rho_\e^v(\Xi_{(x',\xi')} -\eta),  \  \int \rho_\e^s(x'-y)dx'=1 \ \text{and} \  dX_{(x',\xi')} d\Xi_{(x',\xi')}=dx'd\xi',
$$
the last being a consequence of the Hamiltonian structure of the system of characteristics, the first observation is that 
$$\begin{array}{rl}
\int \rho_\e(x',\xi',y, \eta,t) \sgn (\xi') d\xi'dx' & = \dis \int  \rho_\e(x',\xi',y, \eta,t) \sgn (\Xi_{(x',\xi')}) d\xi'dx'
\\[5pt]
& = \dis \int \rho_\e(x',\xi',y, \eta,0) \sgn{\xi'} d\xi'dx'=  \dis \int \sgn{\xi'}\rho_\e^v(\xi' -\eta) d\xi'.  
\end{array}$$

\smallskip

\noindent The second is 

$$ D_\eta \int_\R \sgn(\xi') \rho_\e^v(\xi' -\eta)d \xi' =  2 \rho_\e^v(-\eta) \ \text{and } \  \int_{\R^N} D_y \rho_\e^s(X_{x,\xi} -y) dy=0.$$

\noindent Using the above in the expression for $I_{14}(t)$, after integrating by parts, we find

$$I_{14}(t)= - \int  m^{(1)} (x,\xi,t) \rho^v_\e(-\eta) \rho_\e(x,\xi,y,\eta,t) \p_\xi \Xi_{(x,\xi)}. $$ 

\smallskip

\noindent Repeating this calculation for $I_2(t)$
adding the two inequalities,  we get
\beq\label{takis3}
\begin{array}{rl}
- \f{d}{dt}& [ \rho_\e \star \chi^{(1)} ( y, \eta,t)\;  \rho_\e \star \chi^{(2)} ( y, \eta,t) ]
\\ \\ &\leq 
 \int  [m^{(1)} (x,\xi,t) +  m^{(2)} (x,\xi,t)] \rho_\e^v(-\eta) \rho_\e( x,\xi , y , \eta,t)  
 \p_\xi \Xi_{(x,\xi)} dy d\eta dx d\xi 
 \\[2mm]
 & \quad + {\rm Err^{(1)}_\e}(t) + {\rm Err^{(2)}_\e}(t),
\end{array}
\eeq
where
\beq\label{error1}
\begin{array}{rl}
{\rm Err^{(1)}_\e}(t) 
&=\int  [m^{(1)} (x,\xi,t)  \chi^{(2)} (x',\xi',t) + m^{(2)} (x',\xi',t)  \chi^{(1)} (x,\xi,t)]\\[2mm]
& [  \p_\xi \rho_\e(x,\xi ,y , \eta,t)  \rho_\e(x',\xi',y, \eta,t) +  \rho_\e(x,\xi ,y , \eta,t)  \p_{\xi'} \rho_\e(x',\xi',y, \eta)] dy  d\eta dx d\xi dx' d\xi'.
\end{array}
\eeq
and
\beq\label{error2}
\begin{array}{rl}
{\rm Err^{(2)}_\e}(t) &:=  
\int  [m^{(1)} (x,\xi,t)  \chi^{(2)} (x',\xi',t) \sgn(\xi') + m^{(2)} (x',\xi',t)  \chi^{(1)} (x,\xi,t)\sgn(\xi)]\\[2mm]
&[  \p_\xi \rho_\e(x,\xi ,y , \eta,t)  \rho_\e(x',\xi',y, \eta,t) +  \rho_\e(x,\xi ,y , \eta,t)  \p_{\xi'} \rho_\e(x',\xi',y, \eta)] dy  d\eta dx d\xi dx' d\xi'.
\end{array}
\eeq
\smallskip

\noindent {\it The other  terms in $\f{d}{dt} F_\e$.}  
Since, for each $y$ and $\eta$,  $q_\e(x,\xi,y,\eta,t)$ solves \eqref{rho}, for $i=1,2$, we have 
$$
 \f{d}{dt}  q_\e \star \chi^{(i)} ( y, \eta,t) = -  \int \p_\xi  q_\e  (x,\xi, y, \eta , t) m^{(i)}(x,\xi,t) dx d\xi,
$$

\noindent  and, in view of \eqref{takis1},

$$
 \f{d}{dt}  q_\e \star \chi^{(i)} ( y, \eta,t)= \int m^{(i)}(x,\xi,t) \rho_\e^v(-\eta) \rho_\e(x,\xi,y,\eta,t;t_0)  \p_\xi \Xi_{(x,\xi)} d\eta dy dx d\xi.
$$ 
\smallskip

\noindent {\it The upper bound  $\f{d}{dt} F_\e$.} Combining the previous two steps we conclude that, for almost every $ t\in [0,h]$, 
\beq\label{Error}
\f{d}{dt} F_\e(t) \leq   {\rm Err^{(1)}_\e}(t) +  {\rm Err^{(2)}_\e}(t). 
\eeq

\smallskip

\noindent {\it The estimate of the error terms.} It is immediate that, 
for $0\leq \tau \leq h$ and $i=1, 2$,
$$
\int_0^\tau | {\rm Err^{(i)}_\e}(t)  | dt \leq A(h, \e) \int_0^\tau  \int  [m^{(1)} (x,\xi,t) + m^{(2)} (x,\xi,t)]dx d\xi dt,  
$$
where 
$$
A(h, \e) := \sup_{0\leq t\leq h, x\in \R^N, \xi \in \R} B_\e(x, \xi, t)
$$
and
$$
B_\e  :=     \int \left| \int [ \p_ \xi \rho_\e(x,\xi ,y , \eta,t)  \rho_\e(x',\xi',y, \eta,t) +  \rho_\e(x,\xi ,y , \eta,t) \p_\xi'  \rho_\e(x',\xi',y, \eta)] \;   dy \; d\eta \right| dx' d\xi' ,
$$ 

\smallskip

\noindent We show that there exists a uniform $C>0$ such that 
\beq
A^{(i)}(h, \e)  \leq C \omega(h).
\label{eq:Ate}
\eeq 

\smallskip

\smallskip

\noindent {\it The estimate of $B_\e$.}
To prove \eqref{eq:Ate} for $A_\e$ we observe that, in light of  the properties of $\rho_\e^s$ and  $\rho_\e^v$, 
$$
\int \big[   \p_\xi \rho_\e(x,y,\xi, \eta,t)  \rho_\e(x',\xi',y, \eta) +  \rho_\e(x,y,\xi, \eta,t)  \p_ \xi'  \rho_\e(x',\xi',y, \eta) \big] dy d\eta
$$
$$
= \int [ \p_\xi X \cdot D \rho^s_\e(X-y)  \rho^v_\e(\Xi - \eta)  \rho^s_\e(X'-y )  \rho^s_\e( \Xi' - \eta) +   \rho^s_\e(X-y )  \rho^s_\e( \Xi - \eta) {\p_ \xi X'} \cdot D \rho^s_\e(X'-y)  \rho^v_\e(\Xi' - \eta)  ] dy d\eta
$$
$$
+ \int  [ \rho^s_\e(X-y)  \p_\xi \Xi  \cdot D \rho^v_\e(\Xi - \eta)  \rho^s_\e(X'-y )  \rho^s_\e( \Xi' - \eta) + \rho^s_\e(X-y )  \rho^s_\e( \Xi - \eta) \rho^s_\e(X'-y)  \p_\xi \Xi' \cdot D \rho^v_\e(\Xi' - \eta) ] dy d\eta,
$$
where to simplify the notation we omitted  for $X, \  \Xi, \  X'  \ \text{and} \ \Xi'$.
\smallskip

\noindent  The integrands in the two terms behave similarly, hence, here,  we study the first one. To this end, we first notice that  $\p_\xi X $ being  independent of $y$ and $\eta$ can be factored  out of the integral, while 
$$
\int_{\R^N} [D \rho^s_\e(X-y)  \rho^v_\e(\Xi - \eta)  \rho^s_\e(X'-y )  \rho^s_\e( \Xi' - \eta) +   \rho^s_\e(X-y )  \rho^s_\e( \Xi - \eta) D \rho^s_\e(X'-y)  \rho^v_\e(\Xi' - \eta)]dy= 
$$
$$
= - \int_{\R^N} D_y [\rho^s_\e(X-y)  \rho^s_\e(X'-y ) ]  \rho^v_\e(\Xi - \eta) \rho^s_\e( \Xi' - \eta)dy=0.
$$

\smallskip

\noindent  Therefore the term we need to estimate is 
$$
\left|\left[ \p_\xi X - \p_\xi X' \right] \cdot   \int D \rho^s_\e(X-y)  \rho^v_\e(\Xi - \eta)  \rho^s_\e(X'-y )  \rho^s_\e( \Xi' - \eta) dy d\eta\right|,
$$
which is controlled by 
$$
\left| \p_\xi X - \p_\xi X' \right|  \int | D \rho^s_\e(X-y)|  \rho^v_\e(\Xi - \eta)  \rho^s_\e(X'-y )  \rho^s_\e( \Xi' - \eta) dy d\eta.
$$

\smallskip

\noindent It follows that 

$$ 
A(h, \e) \leq 
\sup_{\mathcal A}  \left| \p_\xi X - \p_\xi X' \right| 
\int | D \rho^s_\e(X-y)|  \rho^v_\e(\Xi - \eta)  \rho^s_\e(X'-y )  \rho^s_\e( \Xi' - \eta) dy d\eta dx' d\xi' ,
$$
where 
$$\mathcal A=: \{ (x,\xi,t) \in \R^N\times\R\times (0,h]: |X(0)-X'(0) |\leq C \e, \  |\Xi (0) - \Xi'(0) |\leq C\e \}.$$
\smallskip

\noindent  Since, in light of the Hamiltonian structure of the system of the characteristics, we have $dx' d\xi' = dX' d\Xi' $, there exists $C>0$ such that 
$$
  \int | D \rho^s_\e(X-y)|  \rho^v_\e(\Xi - \eta)  \rho^s_\e(X'-y )  \rho^s_\e( \Xi' - \eta) dy d\eta dx' d\xi' =  \int | D \rho^s_\e(X-y)|  \rho^v_\e(\Xi - \eta)  dy d\eta \leq C/{\e},
$$
and the estimate is reduced to showing that 
$$
\sup_{\mathcal A} |\p_\xi X - \p_\xi X'| C/\e \leq \omega(h).
$$

\smallskip

\noindent  The last inequality follows from  a cancelation property proved in \cite{LSbook}, which yields  that the characteristics  \eqref{characBack} satisfy, with variable $s =W(t)$,
$$
\left| \p_\xi X(s) - \p_\xi X' (s) \right| \leq C s \e, 
$$
whenever the data at $s=0$ are $\e$ close.
\end{proof}

\section{The existence of stochastic entropy solutions}
\label{sec:ex}
%
\subsection*{The general strategy of the existence.} We consider  a family of approximate problems using smooth local uniform approximations  $W_\e$ of $W$, that is paths $W_\e \in C^1(\R)$ such that, as $\e\to 0$ and for every $T>0$,   $W_\e \to W$  uniformly on 
$[0, T].$  We construct pathwise stochastic entropy solutions as the limit of solutions to \eqref{xscl} with smooth paths.

\noindent Given $u^0 \in (L^1\cap L^\infty)(\R^N)$ the conservation law \eqref{xscl} with $W_\e$ in place of  $W$ has a unique  entropy solution $u_\e$ with kinetic formulation 
\beq \left\{\begin{array}{l}
\p_t \chi_\e + \displaystyle  \dv_x [{\bf a} (x, \xi) \chi_\e  ] \, \dot W_\e -  \dv_\xi [b (x, \xi) \chi_\e  ]\, \dot W_\e   = \p_\xi m_{\e}   \quad  \text{ in } \  \R^N \times \R \times (0,\infty),
\\[3mm]
\chi_\e = \chi(u^0(\cdot), \cdot) \ \text{ on } \quad \R^N\times\R\times\{0\},
\end{array} \right.
\label{eq:skapprox}
\eeq
where, with the notation \eqref{eq:chi}, $\chi_\e(x, \xi,t) = \chi\big(u_\e(x,t) , \xi \big)$  and a measure $m_\e$  satisfying \eqref{measure0} and \eqref{measure} with $\|u_\e(\cdot,t)\|_\infty$ in place of $\|u(\cdot,t)\|_\infty$ and, in principle, $C_{T}$ depending on $\e.$
\smallskip

\noindent To prove that the $u_\e$'s converge, as $\e \to 0$, to a pathwise stochastic entropy solution of the sscl we need to obtain uniform in $\e$ a priori estimates, 
that is to show that, for each $T>0$,  there exist $C_T, K_T>0$, which are independent of $\e$, such that, for all $t \in [0,T]$, 
\beq 
 \int m_\e(x, \xi, t) dx d\xi dt \leq C_T  \| u^0 \|^2_{2} 
  \quad \text{and} \quad  \| u_\e(\cdot, t) \|_{\infty} + \|u_\e(\cdot,t)\|_1 \leq K_T.
\label{ex:apbu}
\eeq
\smallskip

\noindent As mentioned in Section~\ref{sec:kf}, the $L^1$-bound always holds true,  since in light  of \eqref{as:b} and the positivity of $m_\e$, the contraction property for the $\e$ problem yields that 
$$
\f{d}{dt} \int | u_\e(x,t) | dx = \f{d}{dt} \int | \chi_\e(x,t) | dx d \xi = - \dot W_\e (t) \int  b(x,0)  | \chi_\e(x,t) | dx d \xi  -  \int m_\e (x,0, t) dx \leq 0 .
$$
\smallskip

\noindent The $L^\infty$-estimate is more difficult to establish. At the end of this section  we prove a general bound that only uses \eqref{assumptions1} as well as we discuss some examples not covered by it. We also present the bound on the total mass of $m_\e$; we remark that this is a new bound even for  the theory of deterministic inhomogeneous scalar conservation laws.

\subsection*{The limiting process}  In what follows we fix  $T>0$ and work on $[0,T]$. We follow the construction based on weak limits proposed in Perthame \cite{Peuniq, PeKF}, where we refer for many of the technical details,  which relies  on the kinetic formulation and is an alternative to the construction by Young measures in Di Perna \cite{DipernaMVS}. 

\noindent Assuming \eqref{ex:apbu}   
we can extract subsequences such that, as $\e \to 0$,  
$$
u_\e \rightharpoonup u, \quad \chi_\e \rightharpoonup f  \text{in } L^{\infty}\text{weak-$\star$} \quad \text{and } \  m_\e \rightharpoonup m   \ \text{ in $M^1$ weak-$\star$} 
$$
with  $u$, $f$ and $m$ also satisfying  \eqref{measure} and  \eqref{ex:apbu}; here $M^1$ denotes the space of measures on $\R^N\times\R \times (0,\infty).$

\noindent In addition  some elementary structural properties of the nonlinear function $\chi$ in $\xi$ pass to the weak limit and give
\beq
 \sgn f (x, \xi ,t ) = \sgn (\xi), \  |f(x, \xi, t)| \leq 1 \ \text{and } \  \p_\xi  f(x,\xi,t) = \delta(\xi) - \nu(x, \xi, t),  
\label{ex:shape}
\eeq
where $\nu(x, \xi, t) \geq 0$ is the Young measure associated with the weak-$\star$ limit of $u_\e$. 
\smallskip

\noindent Passing to the weak limit in the definition of pathwise stochastic entropy solutions for  \eqref{eq:skapprox} and  using  $\rho_\e(x,\xi,t) = \hat \rho (x, \xi, W_\e(t)-W_\e(t_0))$  as test functions we obtain    
\beq \left\{\begin{array}{l}
d f+ \displaystyle  \dv_x [{\bf a} (x, \xi) f  ] \circ dW -  \dv_\xi [b (x, \xi) f ] \circ dW  = \p_\xi m \, dt  \   \text{ in}  \   \R^N \times \R \times (0,\infty),
\\[3mm]
f = \chi(u^0(\cdot), \cdot) \ \text{ on } \quad \R^N\times\R\times\{0\} .
\end{array} \right.
\label{ex:sk}
\eeq
We leave it up to the reader to check that the definition of pathwise solutions applies to this equation following the lines of Section~\ref{sec:kf}.

\subsection*{The existence result.} The existence theorem is:

\begin{theorem} Assume \eqref{flux}, \eqref{path}, \eqref{as:b}, \eqref{assumptions1} and $u^0 \in (L^1 \cap L^\infty)(\R^N).$  There exists a pathwise stochastic entropy solution $u\in L^\infty \big( (0,T); (L^1\cap L^\infty)(\R^N)\big)$, for all $T>0$,  of \eqref{xscl} which is is continuous at $t=0$ and satisfies \eqref{measure0}  and \eqref{measure}. The solution is given by $u(x,t)= \int_\R f(x, \xi,t) d\xi$ and  $f(x,\xi, t)= \chi(u(x,t), \xi)$ solves \eqref{ex:sk} and satisfies \eqref{ex:shape}. 
\label{th:ex}
\end{theorem}

\proof  
The existence of $u_\e$ has been already discussed and the $L^\infty$- and $L^2$-bounds are proved below.  Therefore, we may pass to the limit as indicated above and obtain the pathwise stochastic entropy solution $f$ of \eqref{ex:sk}. It remains to prove that $f(x,\xi, t)= \chi(u(x,t), \xi)$.

\noindent Consider the functional
$$
G(t) :=  Ê\int [ \; | f ( x,\xi,t)|  -f^2 ( x,\xi,t) ] dx d\xi 
$$
in place of $F(\cdot)$ in Section~\ref{sec:proof} and note that $G(0)=0$ because $f( x,\xi,0)= \chi( x,\xi,0)=\chi(u^0(x),x)$. Following the same proof, we find that $G(t) \leq G(0)=0$ and,  since $|f|\leq 1,$ we conclude that $G\equiv 0$ and thus $f$ takes only the values $0$ or $1$. In other words, in view of  \eqref{ex:shape}, $f$ is an indicator function like $ \chi(u(x,t), \xi)$. 

\noindent The continuity at time $t=0$ follows from a similar procedure (see \cite{PeKF}, Prop.~4.1.7). For a sequence $t_n \to 0$, there is a weak limit $g(x, \xi)$ of $\chi( x,\xi,t_n)$, which satisfies  \eqref{ex:shape} and, by the definition of pathwise stochastic solutions, $g \leq \chi(u^0(x), \xi)$.  This means that $u(t_n)$ converges to $u^0$ strongly. 
\qed

\subsection*{The $L^\infty$-bound.} 
\noindent For $u^0 \in L^\infty(\R^N)$   and $M= \| u^0 \|_\infty$, we build, using the method of characteristics,  a local in time smooth solution $U$ of
\beq\label{upper}
\begin{cases}
U_t + \dv {\bf{A}}(x, U) = 0 \ \text{ in } \  \R^N \times \R, \\[2.5mm] 
U=M \ \text{ on } \ \R\times \{0\}.
\end{cases}
\eeq

\smallskip

\noindent We show below that this smooth solution exists  in $\R^N \times [-\tau, \tau]$ for some $\tau>0$ which does not depend on the size  of $M$. It then follows from the contraction property that $|u_\e(x,t)| \leq  U(x,W_\e(t))$  as long as  $W_\e(t) \in [-\tau, \tau]$. In view of the uniform continuity of the paths,  this last  statement holds for $t\in [0,\tau^*]$ with  $\tau^*$ depending only on $\tau$ and the modulus of continuity of the paths. Then we iterate the argument departing from the constant $\sup_{s\in [-\tau, \tau]}Ê\| U(\cdot,s)\|_\infty$ and built the smooth large solution on $(\tau, 2 \tau)$,  $(2\tau, 3 \tau)$, etc..  After a  finite  number of steps of order $(T/\tau^*+1)$, we reach the final time and obtain  a uniform bound. 

\noindent To construct the smooth solution on $[-\tau, \tau]$ we argue as follows. Departing from a constant $U(x,0)=M$, the smooth solution of  \eqref{upper} 
is built  by the method of characteristics (see \eqref{characF}) with initial condition $Y(0)= x$ and $\zeta(0)=M$ as long as they do not intersect.
This is possible  as long as $\p_x Y(t)$ is invertible.  Notice that   $(\p_x Y, \p_x \zeta)$ solves a system of differential equations (the linearization of \eqref{characF} along $(Y,\zeta)$) with coefficients which, in view  \eqref{assumptions1}, are uniformly bounded independently of $M$.  Since  $\text{det}\p_x Y(0) =1$, the matrix $\p_x Y(t)$ remains invertible for all $t\in [-\tau, \tau]$ for some uniform $\tau>0$. Because the solution is smooth, it generates a solution $U(x, W(t))$ for the stochastic equation \eqref{xscl} on the time interval $(0, \tau^*)$ with $\tau^*>0$  defined in the previous paragraph.
\smallskip

\noindent Next mention some special cases which do not fall in the general theory. 
A classical problem for applications  is when there exist two ordered bounded steady states, that is smooth and bounded $k_\pm$ such that $k_-\leq k_+$  and $\dv A(x, k_\pm (x)) = 0$. If  $k_- \leq u^0  \leq k_+$, then we 
find the estimate
$$
  \| u(t) \|_{\infty} \leq \max \left( \| k_- \|_{\infty}, \| k_+ \|_{\infty} \right).
$$
For instance when $N=1$ and $A(x,u) = c(x) u^2$ with $c(x) \geq c_m >0$, one can choose $k_\pm(x) = \pm \lb_\pm c(x)^{-1/2}$. When $N=2$ and $A(x,u) = {}^\perp\!D V(x)  B(u)$, one can choose constants. Another  example, which is a model  for multiphase flow in a porous medium,  is  $A(x,u) =  V(x)  u(1-u)$ and $k_-=0$, $k_+=1$ give the physical invariant region.

\subsection*{The $L^2$-bound.} For $u^0 \in L^2 (\R^N)$, we prove  the control of the total mass of the measure and employ again an iteration with  uniform time steps.

\noindent We choose $\rho(x,\xi, 0) = \xi$. The solution $\rho(x,\xi,t)$ of the linear transport equation is $\rho(x,\xi,t)= \Xi_{(t, x,\xi)}(0)$, where as before  $(X_{(t,x,\xi)}(0), \Xi_{(t,x,\xi)}(0))$ are the backwards characteristics starting at $t$. From this representation, we  conclude that there exists  $\tau>0$ such that, for $0 \leq t \leq \tau$,  $ \p_\xi \rho(x,\xi, t) \geq \f 12$ and $ \rho(x,\xi, t) \geq \f \xi 2$ for $\xi \geq 0$ and $ \rho(x,\xi, t) \leq \f \xi 2$ for $\xi \leq 0$.

\smallskip

\noindent The definition of stochastic solution yields 
\beq\label{takis10}
\int_0^\tau \int m_{\e}  (x, \xi, t) \p_\xi \rho(x,\xi,t)  dx d\xi dt + \int \rho(x,\xi,\tau) \chi(x,\xi,t) dx d\xi  =  \int \xi \chi(x,\xi,0)dx d\xi = \f 12 \| u^0\|_2^2.
\eeq

\noindent The choice of $\tau$  and the properties of $\rho(x,\xi,t) $  imply  immediately 
$$
\f 12  \int_0^\tau \int m_{\e}  (x, \xi, t)  dx d\xi dt + \f 14\|u(\cdot, \tau)\|_2^2  \leq \f 12  \| u^0\|_2^2 .
$$

\noindent Given a final time $T$, we iterate, as in the proof of the $L^\infty$-bound,  a finite number of times depending on the modulus of $W_\e(\cdot)$, which is, however, independent of $\e$ and we conclude. 

\subsection*{Acknowledgement}  We would like to thank Benjamin Gess for reading  carefully a preliminary version of this note and pointing out several items that needed improvement.

\bibliographystyle{plain}
\bibliography{Biblio}

\bigskip
\bigskip

\noindent ($^{1}$) Coll{\`e}ge de France and CEREMADE, Universit{\'e} de Paris-Dauphine
\\
 1, Place Marcellin Berthelot \\
 75005 Paris Cedex 5, France \\
email: lions@ceremade.dauphine.fr
\\ \\
\noindent ($^{2}$) Sorbonne Universit{\'e}s, UPMC, univ. Paris 06,\\
      CNRS UMR 7598 Laboratoire J.-L. Lions, BC187,\\
       4, place Jussieu,  F-75252 Paris 5\\ 
       and INRIA Paris-Rocquencourt, EPC Mamba \\
email: benoit.perthame@upmc.fr
\\ \\
($^{3}$)  Department of Mathematics \\
             University of Chicago \\
             Chicago, IL 60637, USA \\
            email: souganidis@math.uchicago.edu
\\ \\
($^{4}$)  Partially supported by the National Science Foundation grants DMS-0901802 and DMS-1266383

\end{document}